\documentclass[12pt]{article}

\usepackage{amsmath,amssymb,amsthm, amscd, verbatim, setspace}
\usepackage{graphicx} 
\usepackage{url, hyperref}
\usepackage[section]{placeins}
\usepackage{tabu}
\usepackage{enumerate}
\usepackage{mathtools}
\usepackage[colorinlistoftodos,prependcaption,textsize=tiny]{todonotes}
\usepackage{fullpage}
\usepackage[utf8]{inputenc}
\usepackage[T1]{fontenc}
\usepackage[textheight=646.32pt]{geometry}

\makeatletter
\newtheorem*{rep@theorem}{\rep@title}
\newcommand{\newreptheorem}[2]{%
\newenvironment{rep#1}[1]{%
 \def\rep@title{#2~\ref{##1}}%
 \begin{rep@theorem}}%
 {\end{rep@theorem}}}
\makeatother

\theoremstyle{plain}
\newtheorem{theorem}{Theorem}[section]
\newtheorem{lemma}[theorem]{Lemma}

\newreptheorem{thm}{Theorem}

\newtheorem{proposition}[theorem]{Proposition}
\newtheorem{prop}[theorem]{Proposition}
\newtheorem{corollary}[theorem]{Corollary}
\newtheorem{cor}[theorem]{Corollary}

\newtheorem{defn}{Definition}

\theoremstyle{definition}

\theoremstyle{definition}

\theoremstyle{remark}
\newtheorem{remark}[theorem]{Remark}

\numberwithin{equation}{section}

\makeatletter
\def\subsubsection{\@startsection{subsubsection}{3}%
  \z@{.5\linespacing\@plus.7\linespacing}{.1\linespacing}%
  {\normalfont\itshape}}
\makeatother

\newcommand{\HH}{\mathcal{H}} 
\newcommand{\LL}{\mathcal{L}}

\newcommand{\N}{\mathbb{N}}

\newcommand{\PP}{\mathbb{P}}

\DeclareMathOperator{\Li}{Li}

\usepackage{graphicx}
\usepackage{amssymb}
\usepackage{amsthm}

\begin{document}

\title{Bounded gaps between primes in short intervals
}

\author{Ryan Alweiss and Sammy Luo}

\maketitle

\begin{abstract}
Baker, Harman, and Pintz showed that a weak form of the Prime Number Theorem holds in intervals of the form $[x-x^{0.525},x]$ for large $x$. In this paper, we extend a result of Maynard and Tao concerning small gaps between primes to intervals of this length. More precisely, we prove that for any $\delta\in [0.525,1]$ there exist positive integers $k,d$ such that for sufficiently large $x$, the interval $[x-x^\delta,x]$ contains $\gg_{k} \frac{x^\delta}{(\log x)^k}$ pairs of consecutive primes differing by at most $d$. This
confirms a speculation of Maynard that results on small gaps between
primes can be refined to the setting of short intervals of this
length.

\end{abstract}

\section{Introduction}
The classical Prime Number Theorem gives an asymptotic estimate for $\pi(x)$, the number of primes $ \leq x$.  It can be written in the form

\begin{equation}
\label{eqn:pnt}
\pi(x)=\Li(x)+ O(x\exp(-c(\log x)^{\frac{1}{2}})),
\end{equation}
for some constant $c$, where
\[
\Li(x)=\int_{2}^x \frac{dt}{\log t}\sim \frac{x}{\log x}.
\]
Under the assumption of the Riemann Hypothesis, the error term can be improved to $O(x^{\frac{1}{2}}\log x)$. Assuming this improved bound on the error term, we can estimate the number of primes in a short interval $[x-h,x]$ for $x^{\frac{1}{2}+\epsilon}\leq h \leq x$ and $\epsilon>0$, obtaining
\begin{equation}
\label{eqn:largebound}
\pi(x)-\pi(x-h)\sim h(\log x)^{-1}.
\end{equation}
More generally, if the error term in \eqref{eqn:pnt} could be improved to $O(x^\delta)$, where $0<\delta<1$, we would obtain \eqref{eqn:largebound} for $ x^{\delta+\epsilon} \leq h \leq x$. Since no such improvement is known unconditionally, it is remarkable that results of the form \eqref{eqn:largebound} have nevertheless been shown for some range of values of $\delta<1$. 
The first result of this form is due to Hoheisel \cite{hoheisel},
 who obtained \eqref{eqn:largebound} for $\delta=1-\frac{1}{33000}$. The best range of $\delta$ for which \eqref{eqn:largebound} is currently known is $\delta\in [\frac{7}{12},1]$, a result due to Heath-Brown \cite{hb}.
These results extend readily to the setting of primes in arithmetic progressions. Let $\pi(x;q,a)$ be the number of primes $p\leq x$ such that $p\equiv a\bmod q$. If $\gcd(a,q)=1$, the result corresponding to \eqref{eqn:largebound} is that
\[
\pi(x;q,a)-\pi(x-h;q,a)\sim \frac{h}{\phi(q)\log x},
\]
where $x^\delta \leq h \leq x$, for $\delta > \frac{7}{12}$ and $q\leq (\log x)^A$. 

If we are content with a lower bound on the number of primes in the range $[x-x^\delta,x]$ in place of an asymptotic formula, it is possible to extend these results to smaller values of $\delta$. Heath-Brown and Iwaniec \cite{hbi} showed that
\begin{equation}
\label{eqn:lbprime}
\pi(x)-\pi(x-h)\gg h (\log x)^{-1}
\end{equation}
for $x^\delta \leq h \leq x$ when $\delta > \frac{11}{20}$.

The range of $\delta$ was subsequently improved several times. In 1996, Baker and Harman \cite{1bhp} showed \eqref{eqn:lbprime} for $\delta\geq 0.535$. In 2001, Baker, Harman, and Pintz (BHP) \cite{bhp} further extended the result to all $\delta\geq 0.525$. To date, this remains the best range of $\delta$ for which \eqref{eqn:lbprime} is known.
As an immediate consequence of the work of BHP, we have
\[
p_{n+1}-p_n \ll p_n^{0.525},
\]
where $p_n$ denotes the $n$th prime. This is the best known unconditional upper bound on $p_{n+1}-p_n$ for sufficiently large $n$. (See the work of Ford, Green, Konyagin, Maynard, and Tao \cite{longgaps} for lower bounds on large gaps between primes.)

In this paper, we carry out a suggestion of Maynard and combine the
ideas of BHP with recent advances in the study of small gaps between
primes. To make this precise, we first recall the Twin Prime Conjecture, which asserts that
\begin{equation}
\label{eqn:tpc}
\liminf_{n \to \infty}(p_{n+1}-p_n)=2.
\end{equation}
The Prime Number Theorem implies that the average gap $p_{n+1}-p_n$ between consecutive primes is asymptotic to $\log p_n$.  In 2005, Goldston, Pintz, and Y{\i}ld{\i}r{\i}m (GPY) \cite{gpy} showed that $p_{n+1}-p_n$ can be arbitrarily small compared to $\log p_n$.  Specifically, they proved that
\[
\liminf_{n \to \infty}\frac{p_{n+1}-p_n}{\log p_n}=0.
\]
GPY further showed that, assuming that the primes are sufficiently well-distributed among residue classes modulo most moduli $q$, it is possible to obtain a bound
\begin{equation}
\label{eqn:bg16}
\liminf_{n \to \infty}(p_{n+1}-p_n)\leq 16.
\end{equation}
To be more precise, we say that the primes have \emph{level of distribution} $\theta$ if for any $A>0$,
\begin{equation}
\label{eqn:bv}
\sum_{q\leq x^\theta} \max_{(a,q)=1}\left|\pi(x;q,a)-\frac{\pi(x)}{\phi(q)}\right|\ll_A \frac{x}{(\log x)^A}.
\end{equation}
The celebrated Bombieri-Vinogradov Theorem states that \eqref{eqn:bv} holds for every $\theta<\frac{1}{2}$. This means that the bound on the error term in the Prime Number Theorem for arithmetic progressions given by the Generalized Riemann Hypothesis holds for almost all moduli $q$.
It is conjectured that \eqref{eqn:bv} in fact holds for all $\theta<1$; this conjecture is known as the Elliott-Halberstam Conjecture. 
The methods of GPY show that the left hand side of~\eqref{eqn:tpc} is finite assuming any level of distribution greater than $\frac{1}{2}$. In particular, if we can take $\theta\geq 0.971$, then we have the claimed bound in~\eqref{eqn:bg16}.

The first unconditional proof that the left hand side of~\eqref{eqn:tpc} is bounded was given by Zhang in 2013 \cite{zhang}. 
He obtained the bound
\[
\liminf_{n \to \infty}(p_{n+1}-p_n)< 7\cdot 10^7.
\]
While Zhang's results are inspiring, his methods are highly technical and do not easily generalize. That same year, Maynard \cite{maynard} discovered a way to extend the methods of GPY by modifying the sieve used. He used this to lower the bound obtained by Zhang, as well as give a generalization to gaps between $p_n$ and $p_{n+m}$ for arbitrary fixed $m$. His method obtains the results
\[
\liminf_{n \to \infty}(p_{n+1}-p_n)\leq 600,
\]
and
\[
\liminf_{n \to \infty}(p_{n+m}-p_n)\ll m^3 e^{4m}.
\]
Tao discovered the underlying sieve independently, but arrived at slightly weaker conclusions.

The techniques used in these approaches in fact operate within a more general setting. Say that a set of linear forms $\mathcal{H}=\{L_1,\dots,L_k\}$, where $L_i(n)=a_i n+h_i$, is \emph{admissible} if for every prime $p$ there exists a value of $n$ such that none of the $L_i(n)$ are divisible by $p$. We will often take $a_1=\cdots=a_k=1$, in which case we can think of $\mathcal{H}$ as the set $\{h_1,\dots,h_k\}$.
The goal is then to look for integers $n$ such that many of $L_1(n),\dots,L_k(n)$ are simultaneously prime. Hardy and Littlewood conjectured that for any admissible $\HH$,
\[
\#\{n\leq N\mid a_i n+h_i\text{ prime }\forall i\in [1,k]\} \sim \mathfrak{G}\frac{N}{(\log x)^k},
\]
where $\mathfrak{G}>0$ is an effective constant depending only on $\HH$.
Maynard's results above follow from showing for every admissible set $\HH=\{L_1,\dots,L_k\}$ that for large enough $N$, there is some $n\in [N,2N]$ such that $\gg \log k$ of the integers $L_1(n),\dots,L_k(n)$ are prime.

In \cite{pintz2015}, Pintz extended Maynard's method to prove the lower bound
\begin{align*}
 \#\{n \in [x,2x] & \mid \#(\{L_1(n),...,L_k(n)\}\cap \PP)\geq c \log k,\text{ and }\max_{1\leq i\leq k} P^{-}(L_i(n))\geq n^{c_1(k)}\} \\
& \gg_k \frac{x}{(\log x)^k},
\end{align*}
for some $c>0$ and $c_1(k)>0$. Here $\PP$ is the set of all primes and $P^-(n)$ is the smallest prime factor of $n$ for $n>1$. This makes the work in \cite{maynard} quantitative and strengthens it by ensuring that none of the $L_i(n)$ have small prime factors.

It is natural to consider a localized version of the question of small gaps between primes, looking for such small gaps within an interval $[x-h,x]$, where $x^\delta\leq h\leq x$ for some $\delta<1$. An analogue of the Bombieri-Vinogradov Theorem holds for such intervals under certain restrictions on $\delta$. The form of the statement is
\begin{equation}
\label{eqn:bvshorter}
\sum_{q\leq x^\theta} \max_{(a,q)=1}\left|(\pi(x;q,a)-\pi(x-h;q,a))-\frac{\pi(x)-\pi(x-h)}{\phi(q)}\right|\ll_A \frac{h}{(\log x)^A},
\end{equation}
for $x^{\delta}\leq h \leq x$. The best known result of this form is due to Timofeev \cite{timofeev}, who showed that \eqref{eqn:bvshorter} holds for any $0<\theta<\frac{1}{30}$ and all $x^{\delta}\leq h \leq x$ when $\delta > \frac{7}{12}$. We refer the reader to \cite{pps} for a history of the development of bounds of this form.
Using Timofeev's result, Maynard \cite{maynard2016dense} proved a quantitative, localized analogue of his earlier work on small gaps between primes. Specifically, he obtained that there exists a sufficiently small constant $c_{\delta}>0$ (depending only on $\delta$) such that for all $k>c_{\delta}^{-1}$ for any $\delta > \frac{7}{12}$,
\begin{equation}
\label{eqn:lbadmis}
\#\{n \in [x-h,x] \mid \#(\{L_1(n),...,L_k(n)\}\cap \PP)\geq c_{\delta} \log k\}\gg_{k,\delta} \frac{h}{(\log x)^k},
\end{equation}
for all $h\geq x^\delta$.

The goal of this paper is to shorten the interval $[x-h,x]$ in the result above. For $h\leq x^{\frac{7}{12}}$, \eqref{eqn:bvshorter} is not known for any $\theta>0$. However, a result due to Kumchev \cite{kumchev} gives a Bombieri-Vinogradov type average result for a \emph{lower bound} on the  prime indicator function $1_{\PP}(n)$ over intervals of size at least $x^{0.53}$. By applying the arguments of BHP \cite{bhp} in conjunction with a generalization of Watt's mean-value theorem to all Dirichlet $L$-functions \cite{wattl}, we extend Kumchev's result to all intervals of size at least $x^{0.525}$. Confirming Maynard’s speculation, this allows us to show~\eqref{eqn:lbadmis} for all $h\geq x^{0.525}$, with the additional property that all the $L_i(n)$ involved have no small prime factors, as in Pintz's result from \cite{pintz2015}. Our main theorem is the following.

\begin{theorem}
\label{thm:main}
For every $\delta\in [0.525,1]$, there is a constant $c_\delta>0$ such that for all $k\geq c_\delta^{-1}$, we have
\begin{align*}
\#\{n \in [x-h,x] &\mid \#(\{L_1(n),...,L_k(n)\}\cap \PP)\geq c_{\delta} \log k,\text{ and }\max_{1\leq i\leq k} P^{-}(L_i(n))\geq n^{c_1(k)}\} \\
& \gg_{k,\delta} \frac{h}{(\log x)^k},
\end{align*}
where $x^\delta \leq h\leq x$ and $c_1(k)>0$ is a constant depending on $k$.
\end{theorem}

Setting $a_1=\cdots=a_k=1$ and taking $k$ large enough so that $c_\delta \log k > 1$ yields the following important corollary.

\begin{corollary}
\label{cor:bgshort}
For any $\delta\in [0.525,1]$, there exist some positive integers $k,d$ such that for sufficiently large $x$, the interval $[x-h,x]$ contains $\gg_{k} \frac{h}{(\log x)^k}$ pairs of consecutive primes differing by at most $d$ when $x^\delta \leq h\leq x$.
\end{corollary}

To obtain our results, we follow the strategy suggested by Maynard in
Section~3 of \cite{maynard2016dense}. His idea was to synthesize the results in ~\cite{wattl} and~\cite{kumchev} to exhibit bounded gaps between primes in intervals of the above
length. The main novelty in this paper is verifying that one can indeed combine these results, the details of which are carried out in Section~\ref{sec:kummore}.  In Section~\ref{sec:defs}, we go over some of the background for our results, stating results we will use and reviewing basic notation and definitions used in the remainder of the paper. In Section~\ref{sec:1stresult}, we give estimates of weighted sums that are short interval analogues of the sums appearing in \cite{maynard}, which suffice to show that when $k$ is sufficiently large, every interval $(x-h,x]$ for $h\geq x^{0.525}$ and large enough $x$ contains at least one $n$ for which $\gg \log k$ of the $L_i(n)$ are prime. In Section~\ref{sec:2ndresult}, we prove Theorem~\ref{thm:main} in full.

\section{Notation and Background}
\label{sec:defs}

We will closely follow the notation of Maynard in \cite{maynard}, with a few modifications. Denote by $(a,b)$ and $[a,b]$ the greatest common factor and least common multiple, respectively, of $a$ and $b$. Let $\tau(n)$ denote the number of divisors of $n$, and let $\tau_{r}(n)$ denote the number of ways to write $n$ as the product of an $r$-tuple of positive integers. Let $\phi(n)$ be the Euler totient function and let $\mu(n)$ be the M\"{o}bius function. As mentioned previously, $1_{\PP}(n)$ is the indicator function for the primes, and $P^{-}(n)$ denotes the smallest prime factor of $n$ for $n>1$. We write $n\sim N$ to mean $N/2<n\leq N$ and $n\asymp N$ to mean $c_1N\leq n\leq c_2 N$, where $c_1,c_2\geq 0$ are unspecified absolute constants.

We fix $k$ and the admissible set $\HH=\{L_1,\dots,L_k\}$, where $L_i(n)=a_i n+h_i$ with $(a_i,h_i)=1$. Throughout, $x$ is taken to be a large integer, and $\eta_k$ is a sufficiently small constant in terms of $k$, not necessarily the same in every appearance. Let
\[
W=\prod_{p \le D_0}p.
\]
Unlike in \cite{maynard}, where $D_0$ is taken to be $\log(\log(\log(x)))$,  we define $D_0$
to depend only on $k$ and $\HH$, as in \cite{pintz2015}.
We defer the precise definition of $D_0$ to Section~\ref{sec:1stresult}.
Pick a residue $v_0$ mod $W$ corresponding to an integer $v$ such that each $L_i(v)$ is relatively prime to $W$.  The existence of such a $v$ is guaranteed by admissibility and the Chinese Remainder Theorem. We let $h \ge x^\delta$ for some $\delta\in [0.525,1]$, and let $R=x^{\frac{\theta}{2}-\epsilon}$ for some small $\epsilon>0$, where $\theta>0$ will be defined later.

We now describe the approach of GPY and Maynard. Define the expressions
\[
S_1(x_1,x_2)=\sum_{\substack{x_1<n \leq x_2\\ n \equiv v_0 \bmod W}} w_n,
\]
and
\[
S_2(x_1,x_2) = \sum_{m=1}^k S_2^{(m)}(x_1,x_2),
\]
where
\[
S_2^{(m)}(x_1,x_2)=\sum_{\substack {x_1<n \leq x_2\\ n\equiv v_0\bmod W }}1_p(a_mn+h_m)w(n).
\]
Here $w(n)$ is a nonnegative weight function that will be defined shortly.

The goal in the work of GPY and Maynard is to show that for some $\rho>0$,
\begin{equation}
\label{eqn:gpym}
\sum_{x< n\leq 2x} \Big(\sum_{i=1}^k 1_{\PP}(n+h_i)-\rho_k\Big) = (S_2(x,2x)-\rho S_1(x,2x))w_n > 0,
\end{equation}
for all sufficiently large $x$. Assuming \eqref{eqn:gpym} holds, we know that for some $n\in [x,2x]$, at least $r(k):=\lfloor \rho_k+1\rfloor$ of the numbers $n+h_1,\dots, n+h_k$ are prime. This yields infinitely many $n$ such that at least $r(k)$ of the numbers $n+h_i$ are prime, implying that
\[
\liminf_{n\to \infty} (p_{n+r(k)-1}-p_n) \leq \max_{1\leq i,j\leq k} (h_i-h_j).
\]
Taking $\{h_1,\dots,h_k\}$ to be, for example, the first $k$ primes larger than $k$, we obtain an upper bound of $O(k\log k)$ by the Prime Number Theorem.

The weights $w(n)$ are constructed as follows. Let $F(t_1,\dots,t_k)$ be a smooth function supported on the subset of $[0,1]^k$ with $\sum_{i=1}^k t_i \leq 1$. Let $R=x^{\frac{\theta}{2}-\epsilon}$ for some small fixed $\epsilon>0$, and define
\[
y_{\vec r}=F\Big(\frac{\log r_1}{\log R},\dots,\frac{\log r_k}{\log R}\Big),
\]
where $\vec r=(r_1,\dots,r_k)$, and let
\[
\lambda_{\vec d}=\left(\prod_{i=1}^k \mu(d_i) d_i\right) \sum_{\substack{\vec r\\d_i\mid r_i\forall i\\ (r_i,W)=1\forall i}} \frac{\mu(\prod_{i=1}^k r_i)^2}{\prod_{i=1}^k \phi(r_i)} y_{\vec r}.
\]
We set
\[
w(n)=\Big(\sum_{d_i \mid a_in+h_i \forall i} \lambda_{\vec d} \Big)^2.
\]

With this choice of weights, Maynard gives the following estimates for $S_1(x,2x)$ and $S_2(x,2x)$. Here we have made the dependence of the error bounds on $D_0$ explicit, and highlighted the further appearance of a constant depending on $k$ by writing $O_k\Big(\frac{1}{D_0}\Big)$.

\begin{proposition}[Maynard, {\cite[Proposition~4.1]{maynard}}]
\label{thm:maynard}
With $S_1,S_2$ as defined above, we have
\begin{align*}
S_1(x,2x)&=\frac{\Big(1+O_k\Big(\frac{1}{D_0}\Big)\Big) \phi(W)^k x (\log{R})^k}{W^{k+1}} I_k(F),\\
S_2(x,2x)&=\frac{\Big(1+O_k\Big(\frac{1}{D_0}\Big)\Big) \phi(W)^k x (\log{R})^{k+1}}{W^{k+1}\log{x}}\sum_{m=1}^kJ_k^{(m)}(F),
\end{align*}
provided $I_k(F)\ne 0$ and $J_k^{(m)}(F)\ne 0$ for each $m$, where
\begin{align*}
I_k(F)&=\int_0^1\dotsi \int_0^1 F(t_1,\dotsc, t_k)^2dt_1\dotsc dt_k,\\
J_k^{(m)}(F)&=\int_0^1\dotsi \int_0^1 \left(\int_0^1 F(t_1,\dotsc,t_k)dt_m\right)^2 dt_1\dotsc dt_{m-1} dt_{m+1}\dotsc dt_k.
\end{align*}
\end{proposition}

Defining
\[
M_k=\sup_{F} \frac{\sum_{m=1}^k J_k^{(m)}(F)}{I_k(F)},
\]
the estimates in Proposition~\ref{thm:maynard} allow Maynard to obtain \eqref{eqn:gpym} with
\[
\lfloor \rho_k + 1\rfloor = \left\lceil \frac{\theta M_k}{2}\right\rceil,
\]
where $\theta$ is a level of distribution of the primes. For $k$ sufficiently large, we have $M_k>\log k-2\log \log k-2$, so that 
\[
\lfloor \rho_k + 1\rfloor = \Big\lceil \Big(\frac{\theta}{2}+o(1)\Big)\log k\Big\rceil.
\]
These results generalize readily to the setting where the linear forms $a_i n+h_i$ do not necessarily have $a_i=1$.

A key ingredient in our work is the following result of Kumchev \cite{kumchev} giving a function, which we will denote by $Y(n)$, which is a lower bound for $1_{\PP}(n)$ and satisfies a modified Bombieri-Vinogradov type result.

\begin{theorem}[Kumchev, {\cite[Theorem~1]{kumchev}}]
\label{thm:primelower}
There is an arithmetic function $Y$ with the following properties:
\begin{enumerate}[(i)]
\item if $n$ is an integer in $[2,x)$, then
\[
Y(n)\leq \begin{cases}
1\text{  if }n\text{ is prime,}\\
0\text{  otherwise;}
\end{cases}
\]
\item if $x/2\leq y<x$ and $z_0=x\exp(-3(\log x)^{1/3})$, then
\[
\sum_{y-z_0<n\leq y}Y(n) \gg \frac{z_0}{\log x}
;\]
\item there is an absolute constant $\epsilon>0$ such that if
\[
E_{Y}(y,h;q,a):=\sum_{\substack{y-h<n\leq y \\ n\equiv a\bmod q}} Y(n)-\frac{hz_0^{-1}}{\phi(q)}\sum_{y-z_0<n\leq y}Y(n),
\]
and if
\[
x^{0.53}\leq z\leq x,\qquad Q\leq zx^{-0.53+\epsilon}, 
\]
then for any $A>0$
\[
\sum_{q\leq Q}\max_{(a,q)=1}\max_{h\leq z} \max_{x/2\leq y<x}|E_Y(y,h;q,a)|\ll_A \frac{z}{(\log x)^A}.
\]
\end{enumerate}
\end{theorem}

By applying a result of Harman, Watt, and Wong \cite{wattl}, we can replace the constant $0.53$ in Theorem~\ref{thm:primelower} by $0.525$.

\begin{theorem}
\label{thm:kum525}
There is an arithmetic function $Y$ with the following properties:
\begin{enumerate}[(i)]
\item if $n$ is an integer in $[2,x)$, then
\[
Y(n)\leq \begin{cases}
1\text{  if }n\text{ is prime,}\\
0\text{  otherwise;}
\end{cases}
\]
\item if $x/2\leq y<x$ and $z_0=x\exp(-3(\log x)^{1/3})$, then
\[
\sum_{y-z_0<n\leq y}Y(n) \gg \frac{z_0}{\log x}
;\]
\item there is an absolute constant $\epsilon>0$ such that if
\[
E_{Y}(y,h;q,a):=\sum_{\substack{y-h<n\leq y \\ n\equiv a\bmod q}} Y(n)-\frac{hz_0^{-1}}{\phi(q)}\sum_{y-z_0<n\leq y}Y(n),
\]
and if
\[
x^{0.525}\leq z\leq x,\qquad Q\leq zx^{-0.525+\epsilon}, 
\]
then for any $A>0$
\[
\sum_{q\leq Q}\max_{(a,q)=1}\max_{h\leq z} \max_{x/2\leq y<x}|E_Y(y,h;q,a)|\ll_A \frac{z}{(\log x)^A}.
\]
\end{enumerate}
\end{theorem}

We leave the details of this extension to Section~\ref{sec:2ndresult}.

\begin{remark}
We can take the implied constant in (ii) to be $\geq 1-\beta$ for some $\beta<1$ made explicit in the computations of \cite{kumchev} and \cite{bhp}. For sufficiently small $\epsilon$, we can take $\beta \leq 0.94$.
\end{remark}

\section{Estimates on the Weighted Sums}
\label{sec:1stresult}

In this section we give, for $0.525\leq \delta\leq 1$, a value of $\rho=\rho_{k,\delta}$ such that
\begin{equation}
\label{eqn:rhocond}
S_2(x-h,x)-\rho S_1(x-h,x) > 0,
\end{equation}
when $x^\delta\leq h \leq x$. 
We will give an asymptotic estimate for $S_1(x-h,x)$ as in \cite{maynard}, but for $S_2(x-h,x)$ it will suffice to give a lower bound. For the proofs of the next two propositions, the only condition we need on $D_0$ is that
\[
D_0>\max(\max_{1\leq i<j\leq k}(a_jh_i-a_ih_j), \max_{1\leq i\leq k}h_i, \max_{1\leq i\leq k}a_i).
\]
We will impose an additional condition on $D_0$ at the end of this section, in order to make the error term in~\eqref{eqn:finals2s1bd} sufficiently small.

\begin{prop}
\label{thm:s2bound}
Let $\delta\geq 0.525$ and $x^{\delta} \leq h\leq x$. We have, for sufficiently large $x$,
\[
S_2^{(m)}(x-h,x)\geq \left(1-\beta+O_k\left(\frac{1}{D_0}\right)\right)\frac{h}{W}\frac{\log R}{\log x}\left(\frac{\phi(W)}{W}\log R\right)^k  J_k^{(m)}(F),
\]
where $\beta<1$ is an absolute constant.
\end{prop}

\begin{proof}
We proceed as in \cite{maynard}, with only a few alterations to the argument. We start with the definition
\[
S_2^{(m)}(x-h,x)=\sum_{\substack {x-h<n\leq x\\ n\equiv v_0\bmod W }}1_p(a_mn+h_m)w(n),
\]
where
\[
w(n)=\Big(\sum_{\substack {\vec d:\: d_i\mid a_in+h_i \forall i}} \lambda_{\vec d}\Big)^2.
\]

First note that for $x$ sufficiently large, $d_m|(a_mn+h_m)$ implies $d_m=1$ if $a_mn+h_m$ is prime, so we can replace $w(n)$ with modified weights,
\[
w'(n)=\Big(\sum_{\substack {\vec d:\: d_i\mid a_in+h_i \forall i\\ d_m=1}} \lambda_{\vec d}\Big)^2,
\]
restricting $d_m$ to equal $1$. Since the weights $w'(n)$ are nonnegative, we have the lower bound
\[
S_2^{(m)}(x-h,x)=\sum_{\substack {x-h<n\leq x\\ n\equiv v_0\bmod W }}1_p(a_mn+h_m)w'(n)\geq \sum_{\substack {x-h<n\leq x\\ n\equiv v_0\bmod W }}Y(a_mn+h_m)w'(n).
\]

Now we expand the square in the definition of $w'(n)$ and switch the order of summation, to obtain
\[
\sum_{\substack{\vec d, \vec e\\ d_m=e_m=1}} \lambda_{\vec d}\lambda_{\vec e} \sum_{\substack{ x-h<n\leq x \\ n\equiv v_0\bmod W\\ [d_i,e_i]\mid a_in+h_i \forall i }} Y(a_mn+h_m).
\]
As in \cite{maynard}, for large enough $x$ the only contribution is from terms where $W,[d_1,e_1]$, $\dots$, $[d_k,e_k]$ are pairwise relatively prime. Indeed, we have chosen $v_0$ such that $(a_in+h_i,W)=1$ for all $i$ when $n\equiv v_0\bmod W$. If $[d_i,e_i]$ and $[d_j,e_j]$ have a common prime factor $q$, then $q\mid a_jh_i-a_ih_j$. Because $(a_i,h_i)=(a_j,h_j)=1$, $a_jh_i-a_ih_j$ is nonzero and bounded, so our choice of $D_0$ guarantees that all prime factors of $a_jh_i-a_ih_j$ divide $W$.  In particular $q\mid W$, a contradiction since $q\mid d_i\mid a_in+h_i$. Thus we can apply the Chinese Remainder Theorem to reduce the restriction on $n$ in the inner sum to a single modular restriction
\[
n \equiv b \bmod q,
\]
where $q=W\prod_{i=1}^k [d_i,e_i]$. Therefore,
\[
a_m n + h_m \equiv a_m b + h_m \bmod qa_m.
\]
Set $b'=a_m b + h_m$.

We approximate the resulting inner sum using property~(iii) of Theorem~\ref{thm:kum525}. Let $z_1=(a_m N+h_m)\exp(-3(\log (a_m N+h_m)^{1/3})\ll a_m z_0$, so that we obtain
\begin{align*}
& \sum_{\substack{ x-h<n\leq x \\ n\equiv v_0\bmod W\\ [d_i,e_i]\mid n+h_i \forall i }} Y(a_mn+h_m)=\sum_{\substack{ a_m(x-h)+h_m<n\leq a_mx+h_m \\ n\equiv b' \bmod qa_m }} Y(n) \\
& = \frac{a_mhz_1^{-1}}{\phi(qa_m)}\sum_{a_mx-z_1+h_m<n\leq a_mx+h_m} Y(n) + E_Y(a_mx+h_m,a_mh;a_mq,b'),
\end{align*}
where again
\[ 
E_{Y}(y,h;q,a)=\sum_{\substack{y-h<n\leq y \\ n\equiv a\bmod q}} Y(n)-\frac{hz_0^{-1}}{\phi(q)}\sum_{y-z_0<n\leq y}Y(n).
\]
We let $q'=qa_m$, and note $\phi(q')=a_m\phi(q)$ because all prime factors of $a_m$ divide $W$.
Thus, the above expression becomes
\[
\frac{hz_1^{-1}}{\phi(q)}\sum_{a_mx-z_1+h_m<n\leq a_mx+h_m} Y(n) + E_Y(a_mx+h_m,a_mh;q',b').
\]
Let $X_h=hz_1^{-1}\sum_{a_mx-z_1+h_m<n\leq a_mx+h_m} Y(n)$, which does not depend on $\vec d$ and $\vec e$, so that our lower bound for $S_2^{(m)}$ is
\[
\frac{X_h}{\phi(W)}\sideset{}{'}\sum_{\substack{\vec d, \vec e\\ d_m=e_m=1}} \frac{\lambda_{\vec d}\lambda_{\vec e}}{\prod_{i=1}^k \phi([d_i,e_i])}+\sum_{\substack{\vec d, \vec e\\ d_m=e_m=1}} \lambda_{\vec d}\lambda_{\vec e} E_Y(a_mx+h_m,a_mh;q',a').
\]
The sum in our main term appears exactly as in \cite{maynard}, so the argument there, encapsulated in Theorem~\ref{thm:maynard}, shows that our main term is
\begin{align*}
&\frac{X_h}{\phi(W)}\left(\sum_{\vec u} \frac{(y^{(m)}_{\vec u})^2}{\prod_{i=1}^k g(u_i)}+O_k\left(\frac{(y^{(m)}_{max})^2}{D_0}\left(\frac{ \phi(W)\log R}{W}\right)^{k-1}\right) \right) \\
&= \left(1+O_k\left(\frac{1}{D_0}\right)\right)\frac{X_h}{\phi(W)}\left(\frac{\phi(W)\log R}{W}\right)^{k+1}J_k^{(m)}(F).
\end{align*}
By property~(ii) of Theorem~\ref{thm:kum525}, we have
\[
X_h \geq \left(1-\beta+O_k\left(\frac{1}{D_0}\right)\right) hz_1^{-1}\frac{z_1}{\log (a_m x+h_m)} = \left(1-\beta+O_k\left(\frac{1}{D_0}\right)\right)\frac{h}{\log x}.
\]
Hence, it follows that
\[
\frac{X_h}{\phi(W)}\sideset{}{'}\sum_{\substack{\vec d, \vec e\\ d_m=e_m=1}} \frac{\lambda_{\vec d}\lambda_{\vec e}}{\prod_{i=1}^k \phi([d_i,e_i])} \ge 
\left(1-\beta+O_k\left(\frac{1}{D_0}\right)\right)\frac{h}{W}\frac{\log R}{\log x}\left(\frac{\phi(W)\log R}{W}\right)^{k}J_k^{(m)}(F).
\]

Meanwhile, we can bound our error term as in \cite{maynard}, using (iii)~of Theorem~\ref{thm:kum525} in place of the Bombieri-Vinogradov theorem. We obtain
\begin{align}
\label{error}
& \sum_{\substack{\vec d, \vec e\\ d_m=e_m=1}} \lambda_{\vec d}\lambda_{\vec e} E_Y(a_mx+h_m,a_mh;q',b')\notag\\
& \ll \lambda_{max}^2 \sum_{q< R^2W} \mu(q)^2 \tau_{3k}(q) E_Y(a_mx+h_m,a_mh;q',b')\notag \\
& \ll y_{max}^2 (\log R)^{2k} \sum_{q< R^2W} \mu(q)^2 \tau_{3k}(q) E_Y(a_mx+h_m,a_mh;q',b') .
\end{align}
Let
\[
E_Y^*(x,z;q)=\max_{(a,q)=1}\max_{h\leq z} \max_{x/2\leq y<x}|E_Y(y,h;q,a)|.
\]
As in \cite{maynard}, we use the Cauchy-Schwarz inequality, part (iii)~of Theorem~\ref{thm:kum525}, and the trivial bound 
\[
|E_Y(a_mx+h_m,a_mh;q',b')|\ll \frac{a_mh}{\phi(a_m q)}=\frac{h}{\phi(q)},
\]
to show that \eqref{error} is
\begin{align*}
& \ll y_{max}^2 (\log R)^{2k} \left(\sum_{q<R^2W}\mu(q)^2 \tau_{3k}^2(q)\frac{h}{\phi(q)}\right)^{1/2}\left(\sum_{q<R^2W} \mu(q)^2 E_Y^*(a_mx+h_m,a_mh;a_mq)\right)^{1/2} \\
& \ll_A \frac{y_{max}^2 h}{(\log x)^{A}}.
\end{align*}
Here we can use (iii)~of Theorem~\ref{thm:kum525} as long as $(a_m x)^{0.525}\leq a_mh \leq a_m x$ and $WR^2 \leq hx^{-0.525+\epsilon_0}$, where $\epsilon_0$ is the absolute constant $\epsilon$ from (iii)~of Theorem~\ref{thm:kum525}. If $R \leq x^{\frac{1}{2}(\delta-0.525+\epsilon_0/2)}$, this error term is dominated by already existing error terms.  Therefore,
\[
S_2^{(m)}(x-h,x)\geq \left(1-\beta+O_k\left(\frac{1}{D_0}\right)\right)\frac{h}{W}\frac{\log R}{\log x}\left(\frac{\phi(W)\log R}{W}\right)^{k}J_k^{(m)}(F),
\]
which is the desired lower bound.
\end{proof}

The estimate of $S_1(x-h,x)$ is an even more direct adaptation of the argument from \cite{maynard}.

\begin{prop}
Let $\delta\geq 0.525$ and $x^{\delta} \leq h\leq x$. We have
\[
S_1(x-h,x)=\left(1+O_k\left(\frac{1}{D_0}\right)\right)\frac{h}{W}\left(\frac{\phi(W)}{W}\log R\right)^k I_k(F).
\]
\end{prop}
\begin{proof}
As before, we expand out the square and switch the order of summation, obtaining
\[
S_1(x-h,x)=\sum_{\substack {x-h<n\leq x\\ n\equiv v_0\bmod W }}w(n)=\sum_{\vec d, \vec e} \lambda_{\vec d}\lambda_{\vec e} \sum_{\substack{ x-h<n\leq x \\ n\equiv v_0\bmod W\\ [d_i,e_i]\mid a_in+h_i \forall i }} 1.
\]
The inner sum is now $\frac{h}{W}+O(1)$, as opposed to $\frac{x}{W}+O(1)$ as in \cite{maynard}. Besides this, the proof is identical to the proof in \cite{maynard}. Note that since $R^2 = x^{\delta-0.525+\epsilon_0/2} \ll \frac{h}{(\log x)^A}$ for any $A$, the first error term we obtain is still appropriately bounded.
\end{proof}
With these estimates in hand, finding an appropriate value of $\rho$ is straightforward. Following the argument of Maynard outlined in Section~\ref{sec:defs}, we can achieve
\begin{equation}
\label{eqn:finals2s1bd}
\frac{S_2}{S_1} \geq \left(1-\beta+O_k\left(\frac{1}{D_0}\right)\right) (M_k-\epsilon)\frac{\log R}{\log x}.
\end{equation}
Defining $D_0$ to be sufficiently large with respect to $k$, we can make the $O_k(\frac{1}{D_0})$ term small enough to obtain~\eqref{eqn:rhocond} for $\rho$ satisfying
\[
\lfloor \rho+1\rfloor\geq \Big\lceil \frac{\delta-0.525+\epsilon_0}{2} (1-\beta) M_k \Big\rceil.
\]

\section{Density of Bounded Gaps in Short Intervals}
\label{sec:2ndresult}
In this section we prove Theorem~\ref{thm:main}, which we will now state in a more precise form.

\begin{theorem}
\label{thm:density}
For any positive integer $k$ and any $\delta\in [0.525,1]$, let $m=\lceil \frac{\delta-0.525+\epsilon_0}{2} (1-\beta) M_k \rceil - 1$, where $\epsilon_0$ is the positive constant appearing in Theorem~\ref{thm:kum525}.
There exists a constant $c_1(k)>0$ such that for any admissible set $\mathcal{H}=\{L_1,\dots,L_k\}$ with $L_i(n)=a_i n+h_i$, the set
\[
S(\mathcal{H}):=\left\{n\in \N:\: \sum_{i=1}^k 1_{\PP}(a_in+h_i) \geq m+1, P^{-}\left(\prod_{i=1}^k (a_in+h_i)\right)\geq n^{c_1(k)}\right\}
\]
satisfies $|S(\mathcal{H})\cap [x-h,x]|\gg_k h(\log x)^{-k}$ for all sufficiently large $x$, where $x^{\delta}\leq h \leq x$. 
\end{theorem}

Our argument is analogous to that in \cite[Section~2]{pintz2015}, modified to fit our study of short intervals. We follow the exposition of \cite{patterns}, which considers a similar problem in a slightly different setting. We begin with the following lemma.

\begin{lemma}
For any $1\leq j\leq k$ there exists $\epsilon>0$ such that for any prime $p>D_0$ with $p<R^\epsilon$ we have
\[
S_{1,p}^{(j)}:=\sum_{\substack{x-h<n\leq x \\ n\equiv v_0 \bmod W \\ p|a_jn+h_j}} w_n \ll_k \frac{(\log p)^2}{p(\log R)^2}\frac{h(\log R)^k}{W}.
\]
\end{lemma}

\begin{proof}
The proof is almost identical to the proof of \cite[Lemma~5.1]{patterns}. By symmetry it suffices to show the result for $j=1$. Expanding the square and rearranging the order of summation as usual gives
\[
S_{1,p}^{(1)}=\sum_{\vec d, \vec e} \lambda_{\vec d}\lambda_{\vec e} \sum_{\substack{x-h<n\leq x\\ n\equiv v_0\bmod W\\ [d_i,e_i]\mid a_in+h_i\forall i\\ p|a_1n+h_1}} 1.
\]

As before we have that $W,[d_1,e_1],\dots,[d_k,e_k]$ are pairwise relatively prime. Since $p>D_0$, we have $(W,p)=1$, and by our choice of $D_0$, we also have $([d_i,e_i],p)=1$ for $i\neq 1$. The Chinese Remainder Theorem, as before, gives that the inner sum is
\[
\frac{h}{W[d_1,e_1,p]\prod_{i=2}^k [d_i,e_i]}+O(1),
\]
so that
\[
S_{1,p}^{(1)}=\frac{h}{pW}\sum_{\vec d,\vec e}\frac{\lambda_{\vec d}\lambda_{\vec e}}{\frac{[d_1,e_1,p]}{p}\prod_{i=2}^k [d_i,e_i]}+O\Big(\sum_{\vec d,\vec e}\lambda_{\vec d}\lambda_{\vec e}\Big).
\]
As in \cite{maynard}, we see that the error term is $\ll y_{max}^2 R^2 (\log R)^{4k}\ll_k \frac{h}{(\log x)^A}$ for any $A$. The sum in the main term is independent of the interval $n$ ranges over, so as in \cite{patterns} it is bounded above by
\[
\sum_{\vec d,\vec e}\frac{\lambda_{\vec d}\lambda_{\vec e}}{\frac{[d_1,e_1,p]}{p}\prod_{i=2}^k [d_i,e_i]} \ll_k \left(\frac{\log p}{\log R}\right)^2 (\log R)^k,
\]
so that
\[
S_{1,p}^{(1)} \ll_k \frac{h}{pW}\left(\frac{\log p}{\log R}\right)^2 (\log R)^k=\frac{(\log p)^2}{p(\log R)^2}\frac{h(\log R)^k}{W},
\]
as claimed.
\end{proof}

\begin{lemma}
For any $\epsilon(k)>0$ there exists $c_1(k)>0$ such that for sufficiently large $x$,
\[
S_1^-(x-h,x):=\sum_{\substack{x-h<n\leq x\\ n\equiv v_0\bmod W\\ P^{-}(\prod_{i=1}^k (a_in+h_i))<n^{c_1(k)}}} w_n \leq \epsilon(k)\frac{h(\log R)^k}{W}.
\]
\end{lemma}
\begin{proof}
We follow the proof of \cite[Lemma~5.2]{patterns}. We have
\[
S_1^-(x-h,x)\leq \sum_{j=1}^k\sum_{D_0<p<N^{c_1(k)}}S_{1,p}^{(j)}.
\]
When $c_1(k)\leq \epsilon \frac{\log R}{\log x}$, we can apply the previous lemma and obtain 
\[
S_1^-(x-h,x)\ll_k \frac{h(\log R)^k}{W} \sum_{D_0<p<x^{c_1(k)}} \frac{(\log p)^2}{p(\log R)^2} \ll \frac{h(\log R)^k}{W}\frac{(c_1(k)\log x)^2}{(\log R)^2}.
\]
Picking $c_1(k)$ sufficiently small thus gives
\[
S_1^{-}(x-h,x)\leq \epsilon(k)\left(\frac{h(\log R)^k}{W}\right),
\]
as desired.
\end{proof}
A similar bound for the contribution to $S_2$ of $n$ such that some $a_in+h_i$ have small prime factors follows easily.
\begin{cor}
For any $\epsilon(k)>0$ there exists $c_1(k)>0$ such that for sufficiently large $N$,
\[
S_2^-(x-h,x):=\sum_{\substack{x-h<n\leq x\\ n\equiv v_0\bmod W\\ P^{-}(\prod_{i=1}^k (a_in+h_i))<n^{c_1(k)}}} \sum_{i=1}^k 1_{\PP}(a_in+h_i)w_n\leq \epsilon(k)\frac{h(\log R)^k}{W}.
\]
\end{cor}
\begin{proof}
Since $w_n$ is nonnegative, the triangle inequality gives us
\[
S_2^-(x-h,x)\leq \sum_{\substack{x-h<n\leq x\\ n\equiv v_0\bmod W\\ P^{-}(\prod_{i=1}^k (a_in+h_i))<n^{c_1(k)}}} k w_n = kS_1^-(N),
\]
which is bounded appropriately by the previous lemma.
\end{proof}

\begin{proof}[Proof of Theorem~\ref{thm:density}]
Let $\rho_m$ be such that $\lfloor \rho_m\rfloor = m$. Note that by definition of $S(\mathcal{H})$, we have
\[
0 < \sum_{i=1}^k 1_{\PP}(a_in+h_i)-\rho_m \leq k,
\]
for $n\in S(\mathcal{H})$. 
Also note that for $n\in S(\mathcal{H})$, since the smallest prime factor of each $a_in+h_i$ is at least $n^{c_1(k)}$, each $a_in+h_i$ has a number of divisors bounded in terms of $c_1(k)$ and the $h_i$, so
\[
w_n=\Big(\sum_{\substack{\vec d:\: d_i|a_in+h_i\:\forall i }} \lambda_{\vec{d}}\Big)^2\ll_{c_1(k),\mathcal{H}} \lambda_{max}^2 \ll_k y_{max}^2 (\log R)^{2k}.
\]

Since $y_{max}\ll F_{max}$, and our choice of $F$ only depends on $k$, assuming that our choice of $\epsilon(k)$ and therefore of $c_1(k)$ only depends on $k$, we in fact obtain $w_n \ll_{k,\mathcal{H}} (\log R)^{2k}$, or equivalently,
\[
1 \gg_{k,\HH} \frac{w_n}{(\log R)^{2k}}.
\]
We then have
\begin{align*}
& |S(\mathcal{H})\cap [x-h,x]|=\sum_{\substack{x-h<n\leq x\\ n\in S(\mathcal{H})}} 1 \\
& \gg_{k,\HH} \frac{1}{(\log R)^{2k}}\sum_{\substack{x-h<n\leq x\\ n\in S(\mathcal{H})}} \left(\sum_{i=1}^k 1_{\PP}(a_in+h_i)-\rho_m \right) w_n. 
\end{align*}
Now, define
\[
S_1^+(x-h,x)=S_1(x-h,x)-S_1^-(x-h,x),
\]
\[
S_2^+(x-h,x)=S_2(x-h,x)-S_2^-(x-h,x).
\]
For $n$ satisfying $P^{-}(\prod_{i=1}^k (a_in+h_i))\geq n^{c_1(k)}$, we have $\sum_{i=1}^k 1_{\PP}(a_in+h_i) - \rho_m > 0$ if and only if $n\in S(\HH)$. So, we have that
\begin{align*}
S_2^+(x-h,x)-\rho_m S_1^+(x-h,x) &=\sum_{\substack{x-h<n\leq x\\ n\equiv v_0\bmod W\\ P^{-}(\prod_{i=1}^k (a_in+h_i))\geq n^{c_1(k)}}} \left(\sum_{i=1}^k 1_{\PP}(a_in+h_i) - \rho_m \right) w_n \\
& \leq \sum_{\substack{x-h<n\leq x\\ n\in S(\mathcal{H})}} \left(\sum_{i=1}^k 1_{\PP}(a_in+h_i)-\rho_m \right) w_n.
\end{align*}
Furthermore, we have
\begin{align*}
S_2^+(x-h,x)-\rho_m S_1^+(x-h,x) &= (S_2-\rho_m S_1) + (S_2^-(x-h,x)-\rho_m S_1^-(x-h,x)) \\
& \geq \frac{h}{W}\left(\frac{\phi(W)\log R}{W}\right)^k I_k(F) \Big((1-\beta)\Big(\frac{\theta}{2}-\epsilon_0\Big)(M_k-\epsilon_0)-\rho_m\Big) \\
& + O_k\Big(\frac{h}{W}\left(\frac{\phi(W)\log R}{W}\right)^k\Big) + O\Big(k\epsilon(k)\frac{h(\log R)^k}{W}\Big).
\end{align*}
Since we are picking $D_0$ to only depend on $k$, we have $\frac{\phi(W)}{W}\gg_k 1$ and similarly $\frac{1}{W}\gg_k 1$. Thus, when $\epsilon_0, \epsilon(k)$ are chosen to be small enough based on $k$, we have
\[
S_2^+(x-h,x)-\rho_m S_1^+(x-h,x) \gg_k h(\log R)^k.
\]

Thus, combining with the previous bounds, we obtain
\[
|S(\mathcal{H})\cap [x-h,x]| \gg_{k,\HH} \frac{1}{(\log R)^{2k}} (S_2^+(x-h,x)-\rho_m S_1^+(x-h,x)) \gg_k h(\log R)^{-k},
\]
as claimed.
\end{proof}

\section{Proof of \texorpdfstring{Theorem~\ref{thm:kum525}}{Theorem 2.3}}
\label{sec:kummore}
In this section we outline a proof of Theorem~\ref{thm:kum525}, the extension of Kumchev's main result in~\cite{kumchev} to all exponents $\delta \geq 0.525$.
To do this, we will synthesize the argument of Kumchev in \cite{kumchev}, which shows the result with $0.53$ in place of $0.525$, and the argument of Baker, Harman, and Pintz in \cite{bhp}. To modify the results of \cite{bhp} for use with primes in arithmetic progressions, we replace the use of Watt's theorem by its extension to Dirichlet L-functions in \cite{wattl}. 
For a function $f$ and a character $\chi$ mod $q$, define
\[
E_f(y,h;\chi)=\sum_{y-h<n\leq y } f(n)\chi(n)-\delta(\chi)hz_0^{-1}\sum_{y-z_0<n\leq y}f(n).
\]
The relation of $E_f(y,h;\chi)$ to $E_f(y,h;q,a)$ is analogous to the relation between the functions $\psi(x;\chi)$ and $\psi(x;q,a)$ in the proof of the Prime Number Theorem for arithmetic progressions given in \cite[Chapter~20]{davenport}. As shown in \cite[Section~4.1]{kumchev}, a function $f$ satisfies property~(iii) from Theorem~\ref{thm:kum525} if it satisfies
\[
\sum_{q\sim Q'} \sideset{}{^*} \sum_\chi \max_{h\leq z}\max_{x/2\leq y<x}|E_f(y,h;\chi)|\ll_A \frac{Q' z}{(\log x)^A}
\]
for all $Q'\leq Q$ and $A>0$, as long as the following conditions hold: $|f(n)|\ll d(n)^B$ for some $B>0$, for some $D$ we have $f(n)=0$ if $P^-(n)<D$, and
\[
D\geq x^\eta, \qquad Q\leq \min(D^2 x^{-\eta},DHx^{-\eta}),
\]
for some $\eta>0$. Here the asterisk on the sum over $\chi$ indicates that the sum is restricted to primitive characters modulo $q$.

Define
\[
\psi(n,w)=\begin{cases}
1\qquad \text{if }P^-(n)>w,\\
0\qquad \text{otherwise.}
\end{cases}
\]
Throughout the arguments that follow we make repeated use of Buchstab's identity,
\[
\psi(n,w_1)=\psi(n,w_2)-\sum_{\substack{pm=n\\w_2\leq p<w_1}}\psi(m,p),
\]
where $2\leq w_2<w_1$.
Note that for $n\in(x^{\frac{1}{2}},x]$ we have
\[
\psi(n,x^{\frac{1}{2}})=1_{\PP}(n),
\]
and for $n\leq x^{\frac{1}{2}}$ we have $\psi(n,x^{\frac{1}{2}})=0$.

Our strategy is to apply Buchstab's identity repeatedly to obtain a decomposition
\begin{equation}
\label{eqn:decompform}
\psi(n,x^{\frac{1}{2}})=\sum_{j=1}^k c_j(n)-\sum_{j=k+1}^\ell c_j(n),
\end{equation}
for some nonnegative arithmetic functions $c_j(n)$ satisfying a particular set of properties.

\begin{defn}
\label{def:fine}
Define a decomposition of the form~\eqref{eqn:decompform} to be a \emph{fine} decomposition if for some $r<k$ and $0.524< \delta \leq 1$, the following properties hold:
\begin{enumerate}[(1)]
\item For all $1\leq j\leq \ell$, $c_j(n)\ll d(n)^B$ for some $B$;
\item $c_j(n)=0$ if $P^-(n)<x^{2\delta-1}$;
\item for any $A>0$ we have
\[
\sum_{q\sim Q}\sideset{}{^*}\sum_{\chi} \max_{h\leq z}\max_{y\sim x}|E_{c_j}(y,h;\chi)|\ll Qz(\log x)^{-A},
\]
for $Q\leq zx^{-\delta}$ and $j\in [1,r]\cup [k+1,\ell]$;
\item if $y\sim x$, $h_0=x\exp(-3(\log x)^{\frac{1}{3}})$, and $\delta\geq 0.525-\epsilon$, then
\begin{equation}
\label{eqn:fine4}
\sum_{y-h_0<n\leq y}\sum_{j=r+1}^k c_j(n)\leq (\beta+o(1))\frac{h_0}{\log x},
\end{equation}
where $\beta<1$ is an absolute constant, which we can take to be $0.94$.
\end{enumerate}
\end{defn}

Given a fine decomposition, we can then take, as in \cite{kumchev},
\[
Y(n)=\sum_{j=1}^r c_j(n) - \sum_{j=k+1}^\ell c_j(n).
\]
This satisfies the conditions of Theorem~\ref{thm:kum525} by the same argument as in \cite{kumchev}. Namely, property~(i) of Theorem~\ref{thm:kum525} is satisfied because 
\[
Y(n)\leq \psi(n,x^{1/2})\leq \begin{cases}
1\text{  if }n\text{ is prime,}\\
0\text{  otherwise,}
\end{cases}
\]
as needed. Property~(ii) of Theorem~\ref{thm:kum525} follows, when $\delta\geq 0.525-\epsilon$, from the equation
\[
\sum_{y-h_0<n\leq y}Y(n)=\sum_{y-h_0<n\leq y}\left(\psi(n,x^{\frac{1}{2}})-\sum_{j=r+1}^k c_j(n) \right),
\]
upon applying the Prime Number Theorem and property~(4) above. Finally, by (1)-(3) and the aforementioned argument from Section~4.1 of \cite{kumchev}, property~(iii) of Theorem~\ref{thm:kum525} holds as long as
\[
Q\leq \min(x^{4\delta-2-\eta},Hx^{-\delta-\eta}),
\]
for some sufficiently small $\eta$. It suffices to consider $H\leq x^{3/5+\eta}$.  In this case, since $\delta\geq 0.52+\eta$, the constraint is just $Q\leq Hx^{-0.525-\eta}$.  This gives property~(iii) of Theorem~\ref{thm:kum525} by taking $\delta<0.525-\epsilon$.
Thus, it remains to find a fine decomposition for $\psi(n,x^{\frac{1}{2}})$.

The following sections contain many technical results giving bounds on expressions involving Dirichlet polynomials or weighted sums of the function $\psi(n,w)$, in many cases very similar to results in \cite{bhp} or \cite{kumchev}. Where relevant, significant changes in the proofs from those in \cite{bhp} and \cite{kumchev} are indicated.

\subsection{Dirichlet polynomials}
The lemmas in this section are essentially the same as the lemmas of \cite[Section~2]{bhp} translated for general Dirichlet L-functions, with the addition of a few tools from \cite[Section~2]{kumchev} to deal with the additional factors of $Q$ that appear. They can be seen as strengthened versions of the lemmas of \cite[Section~2]{kumchev}.

We borrow our notation from both \cite{bhp} and \cite{kumchev}, as appropriate. Recall that $\delta\geq 0.525$. 
Let $\LL=\log x$, $\Psi(T)=\min(zx^{-\frac{1}{2}},x^{\frac{1}{2}}T^{-1})$, and $w=\exp(\LL/\log \LL)$. Write $Q=x^\theta$.
$\epsilon$ and $\eta$ are taken to be small constants, not necessarily the same in every appearance. Likewise, $B$ is taken to be a large constant, not necessarily the same in every appearance. When the expression $\LL^{-A}$ appears, $A$ can be taken to be arbitrarily large.
We note that in many of the results below we will impose the condition that $Q\leq zx^{-\delta-\epsilon/2}$. This condition gives the bound
\[
\Psi(T)QT\leq (x^{\frac{1}{2}}T^{-1}) zTx^{-\delta-\epsilon/2}=x^{1-\delta-\epsilon/2}(zx^{-\frac{1}{2}}),
\]
which will be useful for the integral bounds we wish to show.

Define an \emph{$L$-factor} to be a Dirichlet polynomial of the form
\[
\sum_{k\sim K} \chi(k) k^{-s}\qquad \text{or}\qquad \sum_{k\sim K} \chi(k)(\log k) k^{-s}.
\]

We assume without further comment that all Dirichlet polynomials defined in the following results are of the form
\[
M(s,\chi)=\sum_{m\sim M} a_m \chi(m) m^{-s},
\]
where the coefficients $a_m$ are bounded by $(\tau(m))^B$ for some $B$.

When considering integrals over an interval $[U_0,U]$, we define the $L_p$ norms
\[
\|N\|_p:=\begin{cases}
\left(\sum_{q\sim Q}\sideset{}{^*}\sum_{\chi} \int_{U_0}^U |N(\frac{1}{2}+it,\chi)|^p\right)^{1/p}\qquad \text{if }1\leq p<\infty,\\
\sup_{(t,q,\chi):q\sim Q, t\in [U_0,U]}|N(\frac{1}{2}+it,\chi)| \qquad \text{if } p=\infty.
\end{cases}
\]
Note that these standard norms are related by a simple bound to the norms that Kumchev defines in \cite{kumchev} in terms of well-spaced sets. As in \cite{kumchev}, we define a \emph{well-spaced set} $\mathcal{T}=\mathcal{T}(Q',T)$ to be a set of tuples $(t,q,\chi)$ with $|t|\leq T$, $q\sim Q'$ such that if $(t,q,\chi),(t',q,\chi)\in \mathcal{T}$ with $t\neq t'$, then $|t-t'|\geq 1$. By considering for each ordered pair $(q,\chi)$ the optimal choice of $t$ in each unit interval, we obtain the bound
\[
\sum_{q\sim Q'}\sideset{}{^*}\sum_{\chi} \int_{-T}^T \Big|N\Big(\frac{1}{2}+it,\chi\Big)\Big|^p \leq 2 \max_{\mathcal{T}}\sum_{(t,q,\chi)\in\mathcal{T}} \Big|N\Big(\frac{1}{2}+it,\chi\Big)\Big|^p,
\]
so that asymptotic bounds on Kumchev's norms apply to these standard $L_p$ norms as well.
We start by recalling the following result giving a bound on the $L_2$ norm of a generic Dirichlet polynomial.
\begin{lemma}[Kumchev, {\cite[Lemma~1]{kumchev}}]
\label{lem:kum1}
Given a Dirichlet series $N(s,\chi)=\sum_{n\sim N} b_n\chi(n) n^{-s}$, we have
\[
\|N\|_2^2\ll (N+Q'^2 U)G\LL,
\]
where
$G=\sum_{n\sim N}|b_n|^2 n^{-1}$.
\end{lemma}
Since the Dirichlet series we work with always have coefficients bounded by a power of the divisor function, we always have $G\ll N^\varepsilon$ for any $\varepsilon>0$.
In certain special cases, we can obtain stronger bounds on similar norms. The following lemma is an analogue of Lemma~2 from \cite{bhp}, and is essentially a form of the $L$-function analogue of Watt's theorem, proven in \cite{wattl}.

\begin{lemma}
\label{lem:lbhp2}
If $K(s,\chi)$ is an L-factor, $M<x$, $K \leq 4Q'U$, and $Q' \leq \max(K,U)$, then
\[
\sum_{q\sim Q'}\sideset{}{^*}\sum_{\chi}\int_{1/2+iU/2}^{1/2+iU}|M(s,\chi)|^2|K(s,\chi)|^4|ds| \ll (UQ'^2)^{1+\epsilon}M^\epsilon (1+M^2(UQ')^{-1/2}).
\]
\end{lemma}
\begin{proof}
For $K\leq (Q'U)^\frac{1}{2}$, this is essentially proven in \cite{wattl} in the course of proving the main theorem there. 
For $(Q'U)^\frac{1}{2}\leq K\leq 4Q'U$, we proceed as in~\cite{bhp}, this time based on an approximate functional equation for Dirichlet L-functions as stated in \cite{wang}.
Namely, if $s=\frac{1}{2}+it$ and $\chi$ is a primitive character to the modulus $q$, we have for any $X,Y>0$ satisfying $X \gg q$ and $2\pi XY=qt$,
\begin{equation}
\label{eqn:approxfe}
L(s,\chi)=\sum_{n \le X}\chi(n)n^{-s}+A(s,\chi)\sum_{n \le Y}G(n,\chi)n^{s-1}+R(X,Y),
\end{equation}
where
\[
A(s,\chi)=iq^{-s}(2\pi)^{s-1}\Gamma(1-s)e^{\frac{-\pi is}{2}}\chi(-1) \ll q^{-1/2},
\]
\[ 
G(n,\chi)=\sum_{r=1}^{q}\chi(r)\exp(rn/q)=\overline{\chi(n)}G(1,\chi),
\] 
and 
\[
R(X,Y) \ll q^{1/2}X^{-1/2}\log(Y+q+2)+Y^{-1/2}.
\]
If $Q'<U$, then $K \ge \sqrt{Q'U} \ge Q'$, so since $Q' \le \max(U,K)$, it follows for $X=K,2K$ that $R(X,Y) \ll \log(1+Q'U)$. Arguing as in \cite{bhp} then gives the bound
\[
|K|\ll |K'|+|E|,
\]
for some $L$-factor $K'$ with $K'\leq (Q'U)^{\frac{1}{2}}$ and some error $E\ll \log (1+Q'U)$. So, $|K|^4\ll |K'|^4+|E|^4$. Applying Lemma~\ref{lem:kum1} yields
\begin{align*}
& \sum_{q\leq Q'}\sideset{}{^*}\sum_{\chi}\int_{1/2+iU/2}^{1/2+iU}|M(s,\chi)|^2|E|^4|ds| \ll \log(1+Q'U)^4 \sum_{q\leq Q'}\sideset{}{^*}\sum_{\chi}\int_{1/2+iU/2}^{1/2+iU}|M(s,\chi)|^2|ds| \\
& \ll (Q'U)^{\epsilon} (M+Q'^2 U) M^\epsilon,
\end{align*}
which is absorbed into the claimed bound. By applying the argument in the previous case to $K'$, the lemma follows in this case as well.

\end{proof}

The lemmas that follow will make reference to a particular kind of Dirichlet polynomial, those defined by
\[
N(s,\chi)=\sum_{p_i\sim P_i} \chi(p_1\cdots p_u) (p_1\cdots p_u)^{-s},
\]
where $u\leq B$ for some constant $B$, $P_i\geq w$ for all $i$, and $P_1\cdots P_u \leq x$. We call such polynomials ``of bounded product type.''

The following lemma is an analogue of \cite[Lemma~1]{bhp}, and its proof is essentially the same, using a variant of Heath-Brown's identity \cite{6bhp} for $L$-functions instead of for the zeta function.
\begin{lemma}
\label{lem:lbhp1}
Let $N(s,\chi)$ be a Dirichlet polynomial of bounded product type. Then for $\text{Re}(s)=\frac{1}{2}$,
\[
|N(s,\chi)| \le g_1(s,\chi)+ \cdots +g_r(s,\chi),\: \text{with }r \le \LL^B,
\]
where each $g_i$ is of the form
\[
\LL^B \prod_{i=1}^{h}|N_i(s,\chi)|,\qquad \text{with }h \le B,\: \prod_{i=1}^{h}N_i \le x,
\]
and among the Dirichlet polynomials $N_1, \cdots, N_h$ the only polynomials of length greater than $(QT)^{1/2}$ 
are L-factors.
\end{lemma}

Let $\alpha'=\max(\alpha,\theta+(1-\delta))$ from this point onward.
The next two lemmas are direct analogues of Lemmas~3 and~4 of \cite{bhp}. We will go through the proofs in some detail to highlight the appearance of the $\theta$ terms introduced by working with Dirichlet characters and the small extent to which they affect the bounds.

\begin{lemma}
\label{lem:lbhp3}
Let $MN_1N_2K=x$, $T\leq x$. Suppose that $M$, $N_1$, and $N_2$ are of bounded product type, and $K(s)$ is an L-factor. Further suppose that $Q\leq \min(zx^{-\delta-\epsilon/2},T)$.  Let $M=x^{\alpha}$ and $N_j=x^{\beta_j}$ for $j=1,2$. Suppose that

\begin{equation}
\label{eqn:bhp3cond1}
\alpha \le \delta+\theta,
\end{equation}

\begin{equation}
\label{eqn:bhp3cond2}
\alpha'+\beta_1+\frac{1}{2}\beta_2 \le \frac{1}{2}(1+\delta)+\frac{3}{2}\theta,
\end{equation}

\begin{equation}
\label{eqn:bhp3cond3}
\alpha'+\beta_2 \le \frac{1}{4}(1+3\delta)+\theta,
\end{equation}

\begin{equation}
\label{eqn:bhp3cond4}
\alpha'+\beta_1+\frac{3}{2}\beta_2 \le \frac{1}{4}(3+\delta)+\frac{3}{2}\theta.
\end{equation}

Then for $1 \le U \le T$ and $1\leq Q'\leq Q$,
\[
\Psi(T)\sum_{q\sim Q'}\sideset{}{^*}\sum_\chi \int_{U/2}^{U}|(MN_1N_2K)(\frac{1}{2}+it,\chi)|dt \ll Q z \LL^{-A}.
\]
\end{lemma}

Note that this lemma immediately gives a bound for the same sum over all $q\leq Q$ via a dyadic decomposition of $[1,Q]$.

\begin{proof}
The case $Q'=1$ can be dealt with as in \cite[Lemma~3]{bhp}, giving an upper bound of $x^{1/2}\LL^{-A}$. So, we can assume that $Q'>1$.
First consider the case where $K>4Q'U$ and let $N=MN_1N_2$. By Lemma~6 of \cite{kumchev}, we have
\[
K\Big(\frac{1}{2}+it,\chi\Big) \ll \delta(\chi) K^\frac{1}{2}U^{-1}+K^{-\frac{1}{2}}(qU)^\frac{1}{2}\log(qU) \ll \delta(\chi) K^\frac{1}{2}U^{-1}+\LL.
\]
Since $q>1$, we then have $\|K\|_\infty \ll \LL$, so by H\"{o}lder's inequality and Lemma~\ref{lem:kum1}, the sum we wish to bound is at most
\begin{align*}
\|KN\|_1 & \leq \|K\|_\infty \|1\|_2 \|N\|_2 \ll (Q'^2 U)^{\frac{1}{2}}(N+Q'^2U)^{\frac{1}{2}} U^{\epsilon/4} \LL^B \\
& \ll Q^2 T^{1+\epsilon/4} \LL^B + (Q'^2 U \frac{x}{K})^{\frac{1}{2}} U^{\epsilon/4} \LL^B
\ll Qx^{1-\delta-\frac{\epsilon}{4}} \LL^{B}(zx^{-\frac{1}{2}-\epsilon/2}\Psi(T)^{-1}) + Q'^{\frac{1}{2}}U^{\epsilon/4}x^{\frac{1}{2}}\LL^B \\
& \ll Qz\LL^{-A}\Psi(T)^{-1}.
\end{align*}

The last line follows because we can pick $\epsilon$ small enough such that $Q'^{\frac{1}{2}}U^{\epsilon/4}\ll Q^{\frac{1}{2}}x^{\epsilon/4} \ll x^{\theta}\LL^{-A}=Q\LL^{-A}$.

Now take $K\leq 4Q'U$. The Cauchy-Schwarz inequality yields that
\[
\|KN\|_1 \leq \|M\|_2 \|N_1N_2^{\frac{1}{2}}\|_4 \|KN_2^{\frac{1}{2}}\|_4,
\]
where the integrals implicit in the norms are over $s=\frac{1}{2}+it$ with $t\in [U/2,U]$. If $Q'\leq U$ or $Q'\leq K$, then Lemma~\ref{lem:lbhp2} gives the bound
\begin{equation}
\label{eqn:lbhplem3cs}
\|KN_2^{\frac{1}{2}}\|_4^4 \ll (UQ'^2)^{1+\epsilon}M^\epsilon(1+M^2(UQ')^{-1/2}) \ll (Q^2 T)^{1+\epsilon} (1+N_2^2 (QT)^{-1/2}).
\end{equation}
On the other hand, if $K,U\leq Q'$, then Lemma~\ref{lem:kum1} gives
\[
\|KN_2^{\frac{1}{2}}\|_4^4=\|K^2 N_2\|_2^2 \ll (K^2 N_2 + Q'^2 U)x^\epsilon \ll (Q'^2 N_2 + Q'^3)x^\epsilon \ll (Q^2 T + N_2^2 Q(QT)^{1/2}) x^\epsilon,
\]
since $Q'\leq T$. In either case, applying Lemma~\ref{lem:kum1} to the remaining terms in~\eqref{eqn:lbhplem3cs} then gives, for arbitrarily small $\epsilon>0$,
\begin{align*}
\|KN\|_1
& \ll x^{\epsilon/50}(M+Q^2 T)^{1/2}(N_1^2 N_2+Q^2 T)^{1/4}(Q^2 T)^{1/4} (1+N_2^2 (QT)^{-1/2})^{1/4} \\
& \ll \max(1,zx^{-\frac{1}{2}-\epsilon/2}\Psi(T)^{-1})\max(M,Qx^{1-\delta})^{1/2}\max(N_1^2 N_2,Qx^{1-\delta})^{1/4} \\
& \qquad \max(Qx^{1-\delta},N_2^2x^{(1-\delta)/2})^{1/4} \\
& \ll x^\gamma \max(1,zx^{-\frac{1}{2}-\epsilon/2}\Psi(T)^{-1}),
\end{align*}
where
\[
\gamma=\frac{1}{2}\alpha'+\frac{1}{4}\max(2\beta_1+\beta_2,\theta+(1-\delta))+\frac{1}{4}(\theta+(1-\delta))+\frac{1}{4}\max(0,2\beta_2-\frac{1}{2}(1-\delta))-\frac{1}{25}\epsilon.
\]
Conditions (\ref{eqn:bhp3cond1})-(\ref{eqn:bhp3cond4}) guarantee that $\gamma\leq \frac{1}{2}+\theta-\frac{1}{25}\epsilon$, so we obtain that
\[
\Psi(T)\|KN\|_1 \ll Qx^{\frac{1}{2}-\frac{1}{25}\epsilon}\max(zx^{-\frac{1}{2}}, zx^{-\frac{1}{2}-\epsilon/2}) \ll Qz\LL^{-A},
\]
as needed.
\end{proof}

\begin{lemma}
\label{lem:lbhp4}
The conclusion of Lemma~\ref{lem:lbhp3} still holds if hypotheses (\ref{eqn:bhp3cond2})-(\ref{eqn:bhp3cond4}) are replaced by the following:
\begin{enumerate}[(i)]
\item Either $\beta_1 \leq \frac{1}{2}\theta+\frac{1}{2}(1-\delta)$ or $N_1$ is an L-factor;
\item \[\beta_2 \le \frac{1}{8}(1+3\delta)-\frac{1}{2}\alpha'+\frac{1}{2}\theta. \]
\end{enumerate}
\end{lemma}
\begin{proof}
If either $K>4Q'U$, or $N_1>4Q'U$ and $N_1$ is an L-factor, then the argument from the beginning of the proof of Lemma~\ref{lem:lbhp3} still applies. So, assume we are not in these cases. Applying Cauchy-Schwarz in a slightly different way and then proceeding as before, we have
\begin{align*}
\|KN\|_1 & \leq \|M\|_2 \|N_1\|_4 \|KN_2\|_4 \\
& \ll x^{\epsilon/50}(M+Q^2 T)^{1/2}\|N_1\|_4 (Q^2 T)^{1/4} (1+N_2^4 (QT)^{-1/2})^{1/4}.
\end{align*}
Here we either have $\beta_1\leq \frac{1}{2}\theta + \frac{1}{2}(1-\delta)$, in which case Lemma~\ref{lem:kum1} gives
\[
\|N_1\|_4=\|N_1^2\|_2^{1/2}\ll (N_1^2+Q'^2 T)^{1/4} \ll Q^{1/4}x^{\frac{1-\delta}{4}}x^{\epsilon/4}\max(1,zx^{-\frac{1}{2}-\epsilon/2}\Psi(T)^{-1})^{1/4},
\]
or $N_1$ is an L-factor with $N_1\leq 4Q'U$, in which case applying Lemma~\ref{lem:lbhp2} or Lemma~\ref{lem:kum1} depending on whether $Q'\leq \max(K,U)$ again gives
\[
\|N_1\|_4 \ll (Q'^2 T)^{1/4}(1+(Q'T)^{-1/2})^{1/4} x^{\epsilon/4} \ll Q^{1/4}x^{\frac{1-\delta}{4}}x^{\epsilon/4}\max(1,zx^{-\frac{1}{2}-\epsilon/2}\Psi(T)^{-1})^{1/4}.
\]
Thus, $\|KN\|_1\ll x^\gamma \max(1,zx^{-\frac{1}{2}-\epsilon/2}\Psi(T)^{-1})$, where we now let
\[
\gamma=\frac{1}{2}\alpha'+\frac{1}{2}(\theta+1-\delta) - \frac{1}{10}\epsilon + \frac{1}{4}\max(0,4\beta_2-\frac{1}{2}(1-\delta)).
\]
Condition~(ii) then gives $x^\gamma \ll Qx^\frac{1}{2} \LL^{-A}$, and the proof then concludes as before.
\end{proof}

\begin{lemma}
\label{lem:lbhp5}
Let $T\leq x$ and suppose $Q\leq \min(zx^{-\delta-\epsilon/2},T)$. 
Let $K(s)$ be an L-factor. Let $M=x^{\alpha}$, $N=x^{\beta}$, with $KMN=x$, and suppose that $\alpha \le \delta+\theta$ and 
\[
\beta \leq \min\left(\frac{1}{2}(3\delta+1-4\alpha')+2\theta,\frac{1}{5}(3+\delta-4\alpha')+\frac{6}{5}\theta\right).
\]
Suppose further that $M(s)$ and $N(s)$ are of bounded product type. Then
\[
\Psi(T)\sum_{q\leq Q} \sideset{}{^*}\sum_{\chi}\int_{2}^{T}\left|(MNK)\left(\frac{1}{2}+it,\chi\right)\right|dt \ll Q z \LL^{-A}.
\]
\end{lemma}

\begin{proof}
The proof is essentially the same as the proof of \cite[Lemma~5]{bhp}, with $a$ now defined by
\[
a=\max(0,2\beta-(1+\delta)+2\alpha'-3\theta),
\]
and with the invocations of Lemmas~1-4 of~\cite{bhp} replaced by invocations of Lemmas~\ref{lem:lbhp2}-\ref{lem:lbhp4}.

\end{proof}

Note that by symmetry we have the same bound for the integral over the range $[-T,-2]$. Since $M,N$ are of bounded product form and $K$ is an $L$-factor, by \cite[Lemma~5]{2bhp} and \cite[Lemma~6]{kumchev} we have 
\[
\|MNK\|_\infty \ll (MN)^{\frac{1}{2}}\LL^{-A} \ll x^{1/2} \LL^{-A},
\]
so that we can extend the result of Lemma~\ref{lem:lbhp5} to an integral over the whole interval $[-T,T]$. That is,
\[
\Phi(T) \sum_{q\leq Q} \sideset{}{^*}\sum_{\chi}\int_{-T}^{T}|(MNK)(\frac{1}{2}+it,\chi)|dt \ll Q z \LL^{-A}.
\]

\subsection{Sieve estimates}
We use the bounds on Dirichlet polynomials derived in the previous section to obtain the Bombieri-Vinogradov style error bounds we want for sequences under certain conditions on sequence length.

The following lemma, analogous to Lemma~6 of \cite{bhp}, is the key to going from Dirichlet polynomial bounds to these error bounds.
\begin{lemma}
\label{lem:lbhp6}
Let $F(s,\chi)=\sum_{k \sim x}c_k\chi(k) k^{-s}$. If for every $T\leq x$ and $Q\leq \min(zx^{-\delta-\eta},T)$ we have
\[
\Psi(T) \sum_{q\sim Q} \sideset{}{^*} \sum_\chi\int_{-T}^{T} \Big|F\Big(\frac{1}{2}+it,\chi\Big)\Big|dt \ll Q z \LL^{-A},
\]
then
\begin{equation}
\label{eqn:errorbd}
\sum_{q\sim Q} \sideset{}{^*} \sum_\chi \max_{h\leq z}\max_{y\sim x} |E_c(y,h;\chi)|\ll Qz\LL^{-A}.
\end{equation}

\end{lemma}
\begin{proof}
For the term with $q=1$ and $\chi$ the trivial character, \cite[Lemma~6]{bhp} gives a bound of $z\LL^{-A}$. For every other $\chi$, applying the truncated Perron formula and arguing as in \cite[Lemma~9]{kumchev} yields
\[
\max_{h\leq z}\max_{y\sim x} |E_c(y,h;\chi)|\ll \LL \max_{Q\leq T\leq x} \Psi(T) \int_{-T}^T \Big|F\Big(\frac{1}{2}+it,\chi\Big)\Big|dt+x^{\eta},
\] 
which yields the desired result upon summing over $q$ and $\chi$.
\end{proof}

The following lemma is an analogue of Lemma~8 in \cite{bhp}, and the proof carries over with no significant changes.

\begin{lemma}
\label{lem:lbhp8}
Let $M(s)=\sum_{m \sim M}a_m \chi(m) m^{-s}$, $N(s)=\sum_{n \sim N}b_n \chi(n) n^{-s}$, $M=x^{\alpha}$, $N=x^{\beta}$. Suppose that $\alpha \le \delta+\theta-\epsilon$ and
\[
\beta \le \min\Big(\frac{1}{2}(3\delta+1-4\alpha')+2\theta,\frac{1}{5}(3+\delta-4\alpha')+\frac{6}{5}\theta\Big)-2\epsilon.
\]
Finally, suppose that $M$ and $N$ are of bounded product type, and $Q\leq \min(zx^{-\delta-\epsilon/2},T)$.

\[
c(k)=\sum_{\substack{mn\ell=k \\ m\sim M, n\sim N}} a_m b_n \psi(\ell,w),
\]
where $w=\exp\left(\frac{\LL}{\log \LL}\right)$. Then equation (\ref{eqn:errorbd}) holds.
\end{lemma}

The next lemma is analogous to Lemma~12 of \cite{bhp} and Lemma~10 of \cite{kumchev}.

\begin{lemma}
\label{lem:lbhp12}
Let $\alpha \in [0,\frac{1}{2}]$, and write
\[
h=\Big\lceil \frac{\frac{1}{2}-\alpha}{2\delta-1}\Big\rceil,
\]
\[
\alpha^*=\max\left(\frac{2h(1-\delta)-\alpha}{2h-1},\frac{2(h-1)\delta+\alpha}{2h-1}\right).
\]
Suppose
\[
0 \le \beta \le \min(\frac{1}{2}(3\delta+1-4\alpha^{*}), \frac{1}{5}(3+\delta-4\alpha^{*}))-2\epsilon,
\]
and $Q\leq zx^{-\delta-\eta}$. Let $M(s)=\sum_{m \sim M}a_m \chi(m) m^{-s}$, $N(s)=\sum_{n \sim N}b_n\chi_n n^{-s}$, $2M=x^{\alpha}$, $N=x^{\beta}$, where $M(s)$ and $N(s)$ are of bounded product type. Let 
\[ 
I_h=\left[\frac{1}{2}-2h\left(\delta-\frac{1}{2}\right),\frac{1}{2}-(2h-2)\left(\delta-\frac{1}{2}\right)\right),
\] 
and let
\[
\nu(\alpha)=\min\left(\frac{2}{2h-1}(\delta-\alpha),\frac{36\delta-17}{19}\right),
\]
when $\alpha \in I_h$, for $h \ge 1$. Then,  Equation~(\ref{eqn:errorbd}) holds for
\[
c(k)=\sum_{\substack{mn\ell=k \\ m\sim M, n\sim N}} a_m b_n \psi(\ell,x^{\nu}),
\]
for every $\nu\leq \nu(\alpha)$.
\end{lemma}
\begin{proof}
The proof proceeds as in \cite[Lemma~12]{bhp}, with an application of Lemma~\ref{lem:lbhp8} using $x^{\alpha^*}$ in place of $M=x^\alpha$. In place of \cite[Lemma~9]{bhp}, \cite[Lemma~2]{kumchev} is used. As in \cite[Lemma~10]{kumchev}, since $1-\delta \leq \alpha^* \leq \frac{1}{2}$, the dependence on $Q=x^\delta$ is eliminated from the bounds on $\alpha$ and $\beta$.
\end{proof}

An analogue of Lemma~13 in \cite{bhp} follows immediately from the same arguments. It is stated here for ease of reference.

\begin{lemma}
\label{lem:lbhp13}
Suppose that $Q\leq zx^{-\delta-\eta}$. Let $M=x^{\alpha}$, $N_1=x^{\beta}$, $N_2=x^{\gamma}$, where $M(s,\chi),N_1(s,\chi),N_2(s,\chi)$ are of bounded product type. Suppose $\alpha\leq \frac{1}{2}$ and either
\begin{enumerate}[(i)]
\item
\begin{align*}
2\beta+\gamma&\leq 1+\delta-2\alpha^*-2\epsilon,\\
\gamma &\leq \frac{1}{4}(1+3\delta)-\alpha^*-\epsilon,\\
2\beta+3\gamma &\leq \frac{1}{2}(3+\delta)-2\alpha^*-2\epsilon,
\end{align*}
or
\item
\[ 
\beta\leq \frac{1}{2}(1-\theta),\qquad \gamma\leq \frac{1}{8}(1+3\theta-4\alpha^*)-\epsilon.
\]
\end{enumerate}
Let
\[
b_n=\sum_{\substack{n_1n_2=n\\ n_1\sim N_1,n_2\sim N_2}} A_{n_1} B_{n_2},
\]
where $A_{n_1},B_{n_2}$ are the coefficients of $N_1$ and $N_2$. Then equation~\eqref{eqn:errorbd} holds for
\[
c(k)=\sum_{\substack{mn\ell=k \\ m\sim M, n\sim N}} a_m b_n \psi(\ell,x^{\nu}),
\]
for every $\nu\leq \nu(\alpha)$.
\end{lemma}

The last result we need to borrow is an adaptation of Lemma~18 from \cite{bhp}. This gives a bound of the form
\begin{equation}
\label{eqn:lbhp18}
\Psi(T)\sum_{q\leq Q}\sideset{}{^*}\sum_{\chi}\int_{2}^T \left|L_1\left(\frac{1}{2}+it,\chi\right)\cdots L_\ell \left(\frac{1}{2}+it,\chi\right)\right|dt \ll Qz\LL^{-A},
\end{equation}
for $L_1\cdots L_\ell=x$, $\ell\geq 3$, $L_j=x^{\alpha_j}$, $\alpha_j\geq \epsilon$, assuming that $(\alpha_1,\dots,\alpha_\ell)$ lies in a particular region in $[0,1]^l$. The full list of bounds required is omitted for the sake of brevity; it is the same list as in the statement of \cite[Lemma~18]{bhp}, with only the additional condition that $Q\leq zx^{-\delta-\eta}$. The proof is likewise analogous to the proof of \cite[Lemma~18]{bhp}, with factors of $Q$ inserted before occurrences of $T$ in bounds (compare the difference between the proofs of \cite[Lemma~9]{bhp} and \cite[Theorem~4]{2bhp}). An application of Lemma~\ref{lem:lbhp6} then gives us an estimate on sums of the form
\[
\sum_{\substack{p_1\cdots p_{\ell-1}m=n\\p_i\sim L_i \forall i}} \psi(m,p_{\ell-1}),
\]
which will be useful in estimating terms in our final decomposition.

\subsection{Final Decomposition}
We finish the proof of Theorem~\ref{thm:kum525} by outlining the construction of a fine decomposition of $\psi(n,x^{\frac{1}{2}})$.
We start the process by following Kumchev's procedure. Let $w_0=x^{2\delta-1}$, and let $w(m)=x^{\nu(\alpha)}$, where $m=x^\alpha$ and $\nu(\alpha)$ is defined as in the statement of Lemma~\ref{lem:lbhp12}. Note that $w(m)\geq x^{2\delta-1}$ for all $m<x^{\frac{1}{2}}$.

Applying Buchstab's identity twice to $\psi(n,x^{\frac{1}{2}})$ gives
\[
\psi(n,x^{\frac{1}{2}})=\psi(n,w_0)-\sum_{\substack{n=mp\\w_0\leq p<x^{\frac{1}{2}}}} \psi(m,w(p))+\sum_{\substack{n=mp_1p_2\\w(p_1)\leq p_2<p_1<x^{\frac{1}{2}}}} \psi(m,p_2).
\]
We can set
\[
c_1(n)=\psi(n,w_0), \qquad c_{k+1}(n)=\sum_{\substack{n=mp\\w_0\leq p<x^{\frac{1}{2}}}} \psi(m,w(p)),
\]
and
\[
c_2(n)=\sum_{\substack{n=mp_1p_2\\w(p_1)\leq p_2<p_1<x^{\frac{1}{2}}\\p_2<w(p_1 p_2)}} \psi(m,p_2).
\]
As argued in \cite{kumchev}, $c_1,c_2,c_{k+1}$ are nonnegative by construction and satisfy properties~(1)-(3). It remains to further decompose the remaining sum
\begin{equation}
\label{eqn:buchmain}
\sum_{\substack{n=mp_1p_2\\w(p_1)\leq p_2<p_1<x^{\frac{1}{2}}\\p_2\geq w(p_1 p_2)}} \psi(m,p_2).
\end{equation}

Our goal, as in \cite{kumchev}, is to split the range of the sum in \eqref{eqn:buchmain} into several parts, then further decompose the sum over each part using Buchstab's identity. To make sure the decomposition we end with satisfies property~(3) of a fine decomposition, we want most of the terms $c$ to satisfy \eqref{eqn:errorbd}, so that we can set them to be $c_j$ for $j\leq r$ or $j\geq k+1$, depending on whether they are positive or not. To satisfy property~(4), we want the remaining terms $c$, which we set to be $c_j$ for $r+1\leq j\leq k$, to all be nonnegative and to not contribute much to the sum on the left hand side of \eqref{eqn:fine4}. Since all terms in our decomposition are constructed using repeated applications of Buchstab's identity with $w_2\geq w_0\geq x^{2\delta-1}$, properties~(1) and~(2) are automatically satisfied.

We decompose along the same lines as in \cite{bhp}, in order to obtain sharper bounds on the terms contributing to property~(4) than obtained by using Kumchev's decomposition \cite{kumchev}. For any multiset $\mathcal{E}$ of positive integers, the decomposition given in \cite{bhp} is presented as a decomposition of the function $S(\mathcal{E},z)$ defined by
\[
S(\mathcal{E},z)=\sum_{n\in \mathcal{E}} \psi(n,z).
\]
When $c_j$ is given in terms of the function $\psi$, the left hand side of \eqref{eqn:fine4} can be written in terms of the function $S$.
The transformation on sums over $S$ that BHP calls a ``r\^{o}le-reversal'' can also be applied to sums over $\psi$; we refer the reader to \cite[Section~4]{kumchev} for details.

First, we split off the part of \eqref{eqn:buchmain} with $p_1 p_2^2\geq x$. This term is $c_{r+1}(n)$ in the decomposition given by Kumchev, and his analysis shows that its contribution to the left hand side of \eqref{eqn:fine4} is $O\left(\frac{h_0}{\log^2 x}\right)=o(1)\frac{h_0}{\log x}$.

The remainder of the sum is split into six parts exactly as in \cite{bhp}.
Letting $p_1=x^{\alpha_1},p_2=x^{\alpha_2}$, we divide the set of pairs $(\alpha_1,\alpha_2)$ counted in the remaining sum into the regions
\begin{align*}
A:\: & \frac{1}{4}\leq \alpha_1\leq \frac{2}{5},\: \frac{1}{3}(1-\alpha_1)\leq \alpha_2\leq \min(\alpha_1,\frac{1}{2}(3\delta-1),1-2\alpha_1);\\
B:\: & \frac{1}{4}(3-3\delta)\leq \alpha_1\leq \frac{1}{2},\\
& \max(\frac{1}{2}\alpha_1,1-2\alpha_1)\leq \alpha_2\leq \min(\frac{1}{2}(3\delta-1),\frac{1}{2}(1-\alpha_1));\\
C:\: & \nu(0)\leq \alpha_1\leq \frac{1}{3}, \: \nu(\alpha_1)\leq \alpha_2\leq \min(\alpha_1,\frac{1}{3}(1-\alpha_1));\\
D:\: & \frac{1}{3}\leq \alpha_1\leq \frac{1}{2}, \: \nu(\alpha_1)\leq \alpha_2\leq \max(\frac{1}{3}(1-\alpha_1),\frac{1}{2}\alpha_1);\\
E:\: & \frac{1}{2}(3\delta-1)\leq \alpha_1\leq \frac{1}{4}(3-3\delta),\: \frac{1}{2}(3\delta-1)\leq \alpha_2\leq \min(\alpha_1,1-2\alpha_1); \\
F:\: & \frac{1}{3}\leq \alpha_1\leq 2-3\delta,\: \max(1-2\alpha_1,\frac{1}{2}(3\delta-1))\leq \alpha_2\leq \frac{1}{2}(1-\alpha_1).
\end{align*}

As in \cite{bhp}, we note that
$(\alpha_1,\alpha_2)\in A$ if and only if $(1-\alpha_1-\alpha_2,\alpha_2)\in B$,
and similarly for $E$ and $F$. Moreover, in all of these regions, if $\psi(m,p_2)=1$ and $n=mp_1p_2$, then $m$ must be prime. So, the contributions of the regions $A$ and $B$ are equal, and likewise those of $E$ and $F$ are equal.

We handle the sums over these regions using the same procedure as in \cite{bhp}. In place of \cite[Lemmas~12 and~13]{bhp}, we use the analogues Lemmas~\ref{lem:lbhp12} and~\ref{lem:lbhp13}. The result is a contribution of $\leq 0.3$ to the constant factor $\beta$ in \eqref{eqn:fine4} from $A\cup B$. For $E\cup F$, for simplicity of exposition we can place it among the $c_j$ for $r+1\leq j\leq k$. As noted in \cite{bhp}, the total contribution to $\beta$ from this is $\leq 0.09$. For $C$ and $D$, we again follow the argument of \cite{bhp} with the lemma replacements mentioned above where needed. The contribution to $\beta$ from $C$ and $D$ are $< 0.21$ and $<0.34$ respectively. Thus in total we have a fine decomposition with
\[
\beta \leq 0.3+0.09+0.21+0.34 \leq 0.94 < 1,
\]
as wanted.

\section*{Acknowledgements}
This research was conducted at the Emory University REU and was supported by NSF grant DMS-1557690. The authors thank Ken Ono and Jesse Thorner for suggesting the problem and for helpful guidance along the way.

\bibliographystyle{acm}

\bibliography{main}

\end{document}